\documentclass[12pt,a4paper,twoside]{article}
\usepackage[a4paper,left=30mm,right=30mm,top=33mm,bottom=29mm]{geometry}
\usepackage{tikz}
\usetikzlibrary{positioning,arrows}

\usepackage{ifthen}
\usepackage{graphicx}
\usepackage{amsmath,amsfonts,euscript,amssymb,amsthm}
\usepackage{epsfig}
\usepackage{xcolor}

\newcommand{\R}{\mathbb{R}}

\newcommand{\C}{\mathbb{C}}

\newcommand{\eps}{\varepsilon}

\newcommand{\vth}{\vartheta}

\newcommand{\weakto}{\rightharpoonup}

\newcommand{\del}{\partial}

\newcommand{\one}{\mathrm{\bf 1}}

\def\dvect#1#2#3{\left(\begin{matrix}#1\cr #2\cr #3\end{matrix}\right)}

\def\calF{\mathcal{F}}

\def\calS{\mathcal{S}}

\def\calR{\mathcal{R}}

\def\calQ{\mathcal{Q}}

\newtheorem{theorem}{Theorem}[section]
\newtheorem{definition}[theorem]{Definition}
\newtheorem{lemma}[theorem]{Lemma}

\newtheorem{assumption}[theorem]{Assumption}

\newtheorem{remark}[theorem]{Remark}

\numberwithin{equation}{section}

\begin{document}

\thispagestyle{empty}
\begin{center}
  ~\vskip6mm {\Large\bf Representation of solutions to  wave\\[3mm]
    equations with profile functions
  }\\[9mm]

  {\large Agnes Lamacz\footnotemark[1] and Ben Schweizer\footnotemark[2]}\\[4mm]

  May 17, 2019\\[3mm]

\end{center}

\footnotetext[1]{Fakult\"at f\"ur Mathematik, U Duisburg-Essen,
  Thea-Leymann-Stra\ss e 9, 45127 Essen, Germany. {\tt
    agnes.lamacz@uni-due.de}}

\footnotetext[2]{Fakult\"at f\"ur Mathematik, TU Dortmund,
  Vogelspothsweg 87, 44227 Dortmund, Germany. {\tt
    ben.schweizer@tu-dortmund.de}}

\pagestyle{myheadings} \markboth{Representation of solutions to wave
  equations with profile functions}{A. Lamacz and B. Schweizer}

\begin{center}
   \vskip2mm
   \begin{minipage}[c]{0.87\textwidth}
     {\bf Abstract:} Solutions to the wave equation with constant
     coefficients in $\R^d$ can be represented explicitly in Fourier
     space. We investigate a reconstruction formula, which provides an
     approximation of solutions $u(.,t)$ to initial data $u_0(.)$ for
     large times. The reconstruction consists of three steps: 1) Given
     $u_0$, initial data for a profile equation are extracted. 2) A
     profile evolution equation determines the shape of the profile at
     time $\tau = \eps^2 t$. 3) A shell reconstruction operator
     transforms the profile to a function on $\R^d$. The sketched
     construction simplifies the wave equation, since only a
     one-dimensional problem in an $O(1)$ time span has to be
     solved. We prove that the construction provides a good
     approximation to the wave evolution operator for times $t$ of
     order $\eps^{-2}$.\\[-1mm]


   \end{minipage}\\[2mm]
\end{center}


\section{Introduction}

In many applications, one observes solutions of a wave equation that
have the shape of a ring. This can be understood as an effect of large
times: the initial data of the problem are concentrated in a bounded
domain and send waves in every direction. Each of the different waves
travels at the same speed $c$ and after a large time $t$, the
perturbance of the medium is visible mainly at the distance $ct$.  We
observe a ring-like structure (shell-like in three dimensions).

In more mathematical terms, we are interested (in the simplest
setting) in the long time behavior of solutions $u$ to the linear wave
equation
\begin{equation}
  \label{eq:linearwaves}
  \partial^2_t u(x,t) - c^2\Delta u(x,t) = 0\,.
\end{equation}
In this equation, $x\in \R^d$ is the spatial variable and $t\in
[0,\infty)$ is the time variable, the operator $\Delta = \sum_{j=1}^d
\del_{x_j}^2$ acts only on the spatial variables, and $c>0$ is a
prescribed velocity parameter. The equation is complemented with the
initial conditions $u(x,0)=u_0(x)$ and $\partial_t u(x,0)=u_1(x)$.

Our aim is to characterize the shape of solutions in the limit of
large times. We write $\tau = \eps^{2} t$ for a rescaled time
variable. Our result gives an approximate formula for the function
$x\mapsto u(x,\eps^{-2} \tau)$.  The approximate formula is given by a
sequential execution of three operators: An operator $\calR$ extracts
from the initial data $u_0$ initial data for a profile evolution
equation. An evolution operator $J_b$ describes the evolution of the
profile. Finally, a shell operator $\calS$ reconstructs, from a
profile $V$, a shell like solution $u$; the operator $\calS$ maps the
profile $V$ to a function $u$ which looks like $V$ along every ray
through $0$, whereby the profile is centered in the point $ct$.

We present a mathematical proof that the described reconstruction
operator provides, in the large time limit $\eps\to 0$, an
approximation of the solution $u$. The result is based on a stationary
phase method.

The equation \eqref {eq:linearwaves} is a partial differential
equation with constant coefficients on the full space $\R^d$. This
allows to write the solution explicitely in terms of its Fourier
transform. One solution of the wave equation is given by
\begin{equation}
  \label{eq:solutionwaveeqfourier}
  \hat u(k,t) = e^{-ic|k|t}\hat u_0(k)\,,
\end{equation}
another by the same formula upon replacing $-ic|k|t$ by $+ic|k|t$.  In
this work, we always assume that the initial data $u_1$ are such that
the solution $u$ is given by \eqref {eq:solutionwaveeqfourier}.  This
is not a restriction. General initial data can be treated by an
appropriate decomposition, see \cite{Schweizer-Theil-2017} for
details.

\begin{figure}[ht]
  \centering
  \includegraphics[width=0.45\textwidth]{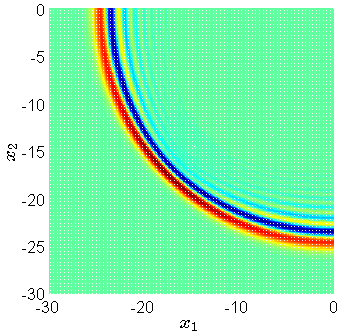}\hfill %
  \begin{tikzpicture}[scale = 0.65]
    \draw[->] (-4.7, 0.0) -- (4.7, 0.0);
    \draw[->] (0.0, -4.7) -- (0.0, 4.7);

    \node[] at (4.5, -.42) {$x_1$};
    \node[] at (0.5, 4.5) {$x_2$};

    \draw (0, 0) circle (3);
    \node[] at (3, 3) {$\{ |x| = ct \}$};

    \draw (0, 0) to (-4,-3);
    \node[] at (-4, -3.4) {$\{ x = zq \}$};
    \end{tikzpicture}
    \caption{Left: The solution to a wave equation. The initial data
      are essentially supported in a unit ball around $x=0$. The wave
      speed of the equation is $c = 1$. After time $t =25$, the
      disturbance of the medium is concentrated in a neighborhood of
      the ring $\{ x\in \R^2 : |x|=ct \}$. The figure shows one
      quadrant and was calculated by T. Dohnal. Right: Sketch for the
      construction of the shell operator $\calS$. A profile $z\mapsto
      V(z;q)$ with a direction parameter $q\in S^{d-1}$ is used along
      the ray $x= zq$; the profile is centered in order to have the
      main pulse near $\{ x\in \R^2 : |x|=ct \}$.}
  \label{fig:num1}
\end{figure}
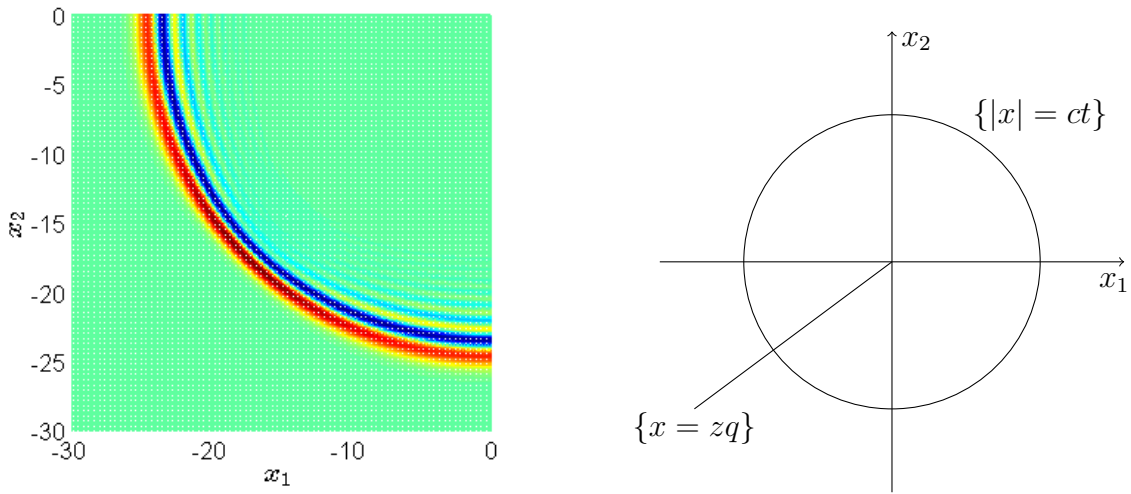

A simplified version of our main result can be stated as follows.  We
define a reconstruction operator in Definition \ref
{def:reconstr-Q}. Essentially, the operator extracts profile
information from $\hat u_0$ and maps the profiles to a shell-like
solution. The shell solution is obtained by using the profiles, in
each direction, and centering them at the distance $c\tau/\eps^2$ from
the origin.

\begin{theorem}[Simplified version of the main theorem] 
  \label{thm:simple-version}
  Let $\hat u_0 :\R^d \to \C$ be smooth initial data with compact
  support with $d\in \{1,2,3\}$.  For arbitrary $\rho>0$, let
  $\hat{\mathcal{Q}}^\rho_0$ be the reconstruction operator of
  Definition \ref {def:reconstr-Q}.  Then, for every $\tau>0$ and
  every $k\in \R^d$ in Fourier space with $|k|>\rho$, the
  reconstruction is similar to the solution of the wave equation: As
  $\eps\to 0$,
  \begin{equation}\label{eq:intro-simple}
    \hat{\mathcal{Q}}^\rho_0 \hat u_0(k,\tau/\eps^2) 
    - e^{-ic|k|\tau/\eps^2}\hat u_0(k) \to 0\,.
  \end{equation}
\end{theorem}

Our main result is stated in Theorem \ref {thm:mainresultQb} below,
and it treats a much more general situation. It allows to treat weakly
dispersive wave equations such as, e.g., $\partial^2_t u(x,t) -
c^2\Delta u(x,t) + \eps^2 B \Delta^2 u = 0$. We only assume that the
solution can be represented in Fourier-space as $\hat u(k,t) = e^{-
  ic|k| t} e^{-ib(|k|) \eps^2 t} \hat u_0(k)$ for some dispersion
function $b = b(|k|)$.  Our theorem yields that the solution $u$ can
be obtained as described above with a shell-like reconstruction from
profiles. In the case of weakly dispersive wave equations, the profile
equation becomes nontrivial: In the leading order case $b(|k|) = b_3
|k|^3$, we obtain a linearized KdV equation $\partial_\tau V(z,q,\tau)
= b_3\partial^3_z V(z,q,\tau)$ for the evolution of the profile
$V(.,q,\tau)$ in the direction $q\in S^{d-1}$. We note that the factor
in the equation is $\eps^2$, since the equation appears as an
effective equation for a problem with micro-structure with length
scale $\eps$, see \cite {DLS14}.

Let us discuss briefly the complexity of the two problems under
consideration. In order to solve a dispersive wave equation on a time
interval of order $\eps^{-2}$, one has to use a computational spatial
domain of order $(\eps^{-2})^d$, the complexity is of the order
$(\eps^{-2})^{d+1}$. To calculate the approximation by the shell
reconstruction operator, one has to extract, for every direction $q\in
S^{d-1}$, a profile function (which is concentrated on a domain of
order $1$. One has to solve (again, for every $q$) a profile evolution
equation on a time interval of order $1$. In the third step, the
profiles are combined to a shell like solution. In particular, the
complexity of the reconstruction process is independent of $\eps$.

\subsection* {Literature}

It is a classical problem to investigate the long time behavior of
solutions to a wave equation. In fact, most research treats more
difficult problem classes than we treat here. We recall that only
linear wave equations with constant coefficients are investigated here; we
assume that the solution can be described in Fourier space by a
multiplication operator that uses the dispersion relation of the
equation.

One of the more difficult problem classes regards homogenization. In
this context, one is interested in a medium that has a periodic
microstructure and asks for the behavior of solutions after long
times. An important contribution in this area is \cite
{SantosaSymes91}; essentially, the second order wave equation in a
heterogeneous medium can be replaced by a weakly dispersive wave
equation in a homogeneous medium. Rigorous results have been obtained
in \cite {DLS14} and \cite {DLS15}, numerical approaches are discussed
in \cite {AbdulleP16}. The same question in a stochastic medium was
addressed in \cite {benoit:hal-01449353}.

Our analysis can be understood as a continuation and improvement of
\cite {Schweizer-Theil-2017}, where the authors studied the long time
behavior for a lattice wave equation. They derived, on the one hand,
that a weakly dispersive wave equation in a homogeneous medium is a
valid replacement for the lattice wave equation. On the other hand,
\cite {Schweizer-Theil-2017} introduced the shell reconstruction
operator; one result regards the approximate reconstruction of the
solution from profiles that are obtained as solutions of a linearized
KdV equation.

The work at hand studies the shell reconstruction operator on a more
abstract level. We do not apply the results to the discrete wave
equation (even though this is possible); we merely investigate an
arbitrary evolution of initial data in Fourier space, where the
evolution is given by harmonic functions through some dispersion
relation. For very general equations, we show that the shell
reconstruction operator provides an approximation of the solution.

We improve the results of \cite{Schweizer-Theil-2017} in two ways. On
the one hand, we can now treat the dimension $d=3$. On the other hand,
we can decouple the effect of dispersion from the analysis of the
shell operator. This makes the analysis more flexible.

An important tool for our method is a stationary phase method. We show
the necessary result in Section \ref {sec.stat-phase}. It regards the
convergence of an oscillatory integral on the sphere. For other
stationary phase results we refer to the book \cite {Wong2001}.

In \cite {MR2060593}, dispersive limit equations are derived for a
linear wave equation in the context of homogenization.  For the long
time behavior of waves in a nonlinear system we mention \cite
{MR3733381}. The monograph \cite {Nedelec-book} contains many
representation formulas for solutions of equations related to the wave
equation.

\section{The reconstruction operator}
\label{sec:recoperator}

We now introduce the three operators that were announced in the
introduction. The concatenation of these operators provides the
reconstruction operator $\calQ$. In the construction, we have to
switch several times between the physical space and the Fourier space.

On the space $X := L^2(\R^d; \C)$ we use the standard
$d$-dimensional Fourier transform $\mathcal{F}_d:X\to X$,
\begin{equation}
  (\calF_d u_0)(k):=\hat u_0(k):=\int_{\R^d} u_0(x)e^{-ik\cdot x}\,dx\,.
\end{equation}
The inverse Fourier transform is $\mathcal{F}_d^{-1}:X\to X$,
\begin{equation*}
  (\calF_d^{-1} \hat u_0)(x)
  :=\frac{1}{(2\pi)^d}\int_{\R^d}e^{ix\cdot k}\hat u_0(k)\,dk\,.
\end{equation*}
By Parseval's identity, $\frac{1}{(2\pi)^{d/2}}\left\|\mathcal{F}_d
  u_0\right\|_{L^2(\R^d)}=\left\|u_0\right\|_{L^2(\R^d)}$.

\medskip
The first operator of our construction has the character of a
restriction: functions on $\R^d$ are mapped to a family of functions
on $\R$ (parametrized by a directional variable $q$).  We use the
space
\begin{equation}
\label{eq:defXs}
X_S:=L^2(\R\times S^{d-1} ; \C)\,, 
\end{equation}
where $S^{d-1}\subset \R^d$ denotes the $(d-1)$-dimensional sphere.

\begin{definition}[The operator $\mathcal{R}$]
  \label{def:restrictionoperator}
  The linear operator $\mathcal{R}$ maps functions on $\R^d$ to
  one-dimensional profiles. We define $\mathcal{R}:X\to X_S$ through
  \begin{equation}\label{eq:restrictionoperator}
    \mathcal{R}\hat u_0(\xi,q)
    := \left(\frac{|\xi|}{2\pi i}\right)^{(d-1)/2}\one_{\{\xi>0\}}\ 
    \hat u_0(|\xi| q)\,. 
  \end{equation}
\end{definition}
It is straightforward to see that
\begin{equation}
\label{eq:ParsevalR}
\|\mathcal{R}\hat u_0\|_{X_S}=(2\pi)^{-(d-1)/2}\|\hat u_0\|_X\,.
\end{equation}
Indeed,
\begin{align*}
  \|\mathcal{R}\hat u_0(\xi,q)\|^2_{X_S}
  &=\int_\R\int_{S^{d-1}}|\mathcal{R}\hat u_0(\xi,q)|^2\,dS(q)d\xi\\
  &=\frac{1}{(2\pi)^{(d-1)}}\int_0^\infty\int_{S^{d-1}}|\xi|^{d-1}
  \left|\hat u_0( |\xi| q)\right|^2\, dS(q)d\xi\\
  &=\frac{1}{(2\pi)^{(d-1)}}\int_{\R^d}|\hat u_0(x)|^2\,dx
  =\frac{1}{(2\pi)^{(d-1)}}\|\hat u_0\|^2_X\,. 
\end{align*}
In order to obtain our results we have to regularize the function
\begin{equation*}
  W(\xi) := |\xi|^{(d-1)/2}\one_{\{\xi>0\}}\,.
\end{equation*}
For a small parameter $\rho>0$ and $d\in\{1,2,3\}$ we consider
functions $W_\rho$ with the following properties: $W_\rho\in
C^{d-1}(\R;\R)$ and
\begin{equation}\label{eq:W-rho}
  \quad W_\rho(\xi)=0\,\,\forall\xi\leq 0,\quad 
  W_\rho(\xi)=|\xi|^{(d-1)/2}\,\,\forall\xi\geq \rho,\quad 
  0\leq W_\rho(\xi)\leq |\xi|^{(d-1)/2}\,\,\forall\xi\geq 0\,. 
\end{equation}
We use the smooth functions $W_\rho$ to define regularized versions of
the operator $\mathcal{R}$.

\begin{definition}[The operator $\mathcal{R}_\rho$]
  \label{def:restrictionoperatorregularized}
  Let $\rho>0$ and let $W_\rho$ be as in \eqref {eq:W-rho}.  The
  linear operator $\mathcal{R}_\rho:X\to X_S$ is defined through
  \begin{equation}\label{eq:restrictionoperator-rho}
    \mathcal{R}_\rho\hat u_0(\xi,q)
    := \left(\frac{1}{2\pi i}\right)^{(d-1)/2} 
    W_\rho(\xi)\hat u_0(|\xi| q)\,. 
  \end{equation}
\end{definition}
As in \eqref{eq:ParsevalR}, by $0\leq W_\rho(\xi)\leq
|\xi|^{(d-1)/2}$, the regularized operators satisfy the estimate
$\|\mathcal{R}_\rho\hat u_0\|_{X_S}\leq(2\pi)^{-(d-1)/2} \|\hat
u_0\|_X$.
 
\medskip
The next operator associates to an initial profile (in Fourier space)
an evolution of profiles (in Fourier space).

\begin{definition}[The operator $J_b$]
  Let $b:\R \to \R$ be a function. The linear operator $J_b$ maps the
  (Fourier transform of) a profile to an evolution of profiles.  We
  define the linear operator $J_b:X_S\to L^\infty(0,\infty;X_S)$
  through
  \begin{equation}
    \label{eq:defJ}
    (J_b \hat V_0)(\xi,q,\tau):=e^{-ib(\xi)\tau}\hat V_0(\xi,q)\,. 
  \end{equation}
\end{definition} 

We emphasize that the time variable is $\tau$ and not $t$, which means
that the evolution of profiles is studied in a new time scale. We will
use $\tau = \eps^2 t$, where $\eps>0$ is a small scaling variable. In
the following, two choices of $b$ will be relevant.

\smallskip {\bf 1)} For $b(\xi)=0$ one has $J_b \hat V_0(\xi,q,\tau) =
J_0 \hat V_0(\xi,q,\tau) = \hat V_0(\xi,q)$. In this case, the time
evolution of the profile is trivial, the profile remains unchanged. In
physical space, this operator describes the trivial evolution equation
$\partial_\tau V(z,q,\tau) = 0$ with initial datum $V(z,q,0) =
(\mathcal{F}_1^{-1} \hat V_0(\cdot,q))(z)$.

\smallskip {\bf 2)} For $b(\xi)=b_3\xi^3$ with $b_3\in \R$ one has
$(J_b \hat V_0)(\xi,q,\tau)=e^{-ib_3\xi^3\tau}\hat V_0(\xi,q)$. With
this choice, the profile evolution in physical space is given by the
linearized KdV-equation $\partial_\tau V(z,q,\tau)=b_3\partial^3_z
V(z,q,\tau)$ with initial data $V(z,q,0)=(\mathcal{F}_1^{-1}\hat
V_0(\cdot,q))(z)$.

\smallskip We note that, independently of the choice of the function
$b$, for every time instance $\tau$, the operator $J_b|_\tau$ is an
isometry because of $|e^{-ib(\xi)\tau}|=1$.

\medskip We finally introduce the shell reconstruction operator. The
reconstruction was also used in \cite {Schweizer-Theil-2017}; it maps
a family of profiles to a ring-like ($d=2$) or shell like ($d=3$)
function on $\R^d$. An important ingredient is the rescaling factor
$(ct)^{-(d-1)/2}$, which has the effect that $L^2$-norms of
reconstructed functions are bounded.

\begin{definition}[The operator $\mathcal{S}$] 
  \label{def:shelloperator}
  We introduce an operator $\mathcal{S}$ that maps profiles to
  functions on $\R^d$.  For a small parameter $\eps>0$ we define the
  linear operator $\mathcal{S}:L^\infty(0,\infty;X_S)\to
  L^\infty(0,\infty;X)$ through
  \begin{equation}
    (\mathcal{S}V)(x,t)
    := \frac{1}{(ct)^{(d-1)/2}}  \one_{\{|x|<2ct\}}\ 
    V\left(|x|-ct, \frac{x}{|x|},\eps^2t\right)\,.
  \end{equation}
\end{definition}

The operator $\mathcal{S}$ constructs, starting from a slowly varying
function $V$, a shell-like solution. The main pulse of the shell-like
solution is near $|x| = ct$ and moves with constant speed $c$; its
profile is given by $V$.  The construction depends on the small
parameter $\eps$, which we suppress in most calculations for the sake
of readability.

\begin{lemma}
  The operator $\mathcal{S}:L^\infty(0,\infty;X_S)\to
  L^\infty(0,\infty;X)$ is bounded. It satisfies, for every $V\in
  L^\infty(0,\infty;X_S)$
  \begin{equation*}
    \|\mathcal SV\|_{L^\infty(0,\infty;X)}
    \leq 2^{(d-1)/2}\, \|V\|_{L^\infty(0,\infty;X_S)}\,.
  \end{equation*}
\end{lemma}

\begin{proof}
  For every $t\in \left(0, \infty\right)$ one has
  \begin{align*}
    \|\mathcal{S}V(\cdot,t)\|^2_{X}&=\int_{\R^d}\frac{1}{(ct)^{d-1}}
    \left|V\left(|x|-ct,\frac{x}{|x|},\eps^2t\right)\right|^2\one_{\{|x|<2ct\}}\,dx\\
    &=\int_0^{2ct}\int_{S^{d-1}}\frac{r^{d-1}}{(ct)^{d-1}}
    \left|V\left(r-ct,q,\eps^2t\right)\right|^2\,dS(q)\,dr\\
    &\leq 2^{d-1}\int_0^{2ct}\int_{S^{d-1}}\left|V\left(r-ct,q,\eps^2t\right)\right|^2\,dS(q)\,dr\\
    &\leq 2^{d-1}\|V(\cdot,\cdot, \eps^2t)\|^2_{X_S}\leq
    2^{d-1}\|V\|^2_{L^\infty(0,\infty;X_S)}\,,
  \end{align*}
  which provides the claim.
\end{proof}

\medskip With the above operators at hand we are now in the position
to introduce our main object of interest, the reconstruction operator
$\mathcal{Q}_b$. It can be described in words as the following
concatenation: From a Fourier transform $\hat u_0$ of initial values,
profile initial data are extracted with the operator $\calR$, then the
profile evolution $J_b$ is applied and the profile is interpreted in
physical space with the inverse Fourier transform
$\mathcal{F}_1^{-1}$. Finally, the shell operator $\calS$ is applied
in order to reconstruct an evolution of functions on $\R^d$.

\begin{definition}[The reconstruction operator $Q_b$]\label{def:reconstr-Q}
  We define the linear reconstruction operator $\mathcal{Q}_b:X\to
  L^\infty(0,\infty;X)$ through
  \begin{equation}
    \mathcal{Q}_b = \mathcal{S}\circ 
    \mathcal{F}_1^{-1}\circ J_b\circ \mathcal{R}\,.
  \end{equation} 
  The operator in Fourier space is denoted as
  $\hat{\mathcal{Q}}_b:=\mathcal{F}_d\circ \mathcal{Q}_b$. For
  $\rho>0$ we define the regularized operators $\mathcal{Q}^\rho_b$
  and $\hat{\mathcal{Q}}_b^\rho$ by replacing $\mathcal{R}$ with
  $\mathcal{R}_\rho$.
\end{definition}

We can now state our main result, which compares two objects. On the
one hand, the solution of a (dispersive) wave equation, which is given
by a multiplication with $e^{-i \left(c|k|/\eps^2+b(|k|)\right)\tau}$
in Fourier space. On the other hand, the reconstruction
$\hat{\mathcal{Q}}_b \hat u_0$. The result is that the two operators
coincide in the limit $\eps\to 0$.

\begin{theorem} [Approximation result for reconstructions]
  \label{thm:mainresultQb}
  Let $\hat u_0\in X$ be continuous initial data with compact support,
  let the dimension be $d\in \{1,2,3\}$, and let $b:\R\to\R$ be a
  dispersion function. Let $\rho$ and $W_\rho$ be as in \eqref
  {eq:W-rho}.  We assume that the regularized profile evolution
  $V^\rho:=(\mathcal{F}_1^{-1}\circ J_b\circ \mathcal{R}_\rho)\hat
  u_0$ satisfies the smoothness and decay properties of Assumption
  \ref{ass:growthcondition}. Then, for every $\tau>0$ and every $k\in
  \R^d$ with $|k|>\rho$, there holds
  \begin{equation}\label{eq:thm-ptw}
    (\hat{\mathcal{Q}}_b^\rho \hat u_0)(k,\tau/\eps^2)e^{ic|k|\tau/\eps^2}
    \to  e^{-ib(|k|)\tau}\hat u_0(k)\,.
  \end{equation}

  Moreover, for every $\tau>0$, there holds weak convergence for the
  non-regularized reconstruction operators,
  \begin{equation}\label{eq:thm-weak}
    (\hat{\mathcal{Q}}_b \hat u_0)(k,\tau/\eps^2) e^{ic|k|\tau/\eps^2}
    \weakto e^{-ib(|k|)\tau}\hat u_0(k)
  \end{equation}
  weakly in $L^2(\R^d)$ as functions in $k\in \R^d$ as $\eps\to 0$.
\end{theorem}

As outlined in Remark \ref{rem:regularityassumptions} below, the
assumptions of Theorem \ref{thm:mainresultQb} are satisfied as soon as
$\hat u_0$ and $b$ are sufficiently smooth.  

The fundamental approximation result of this article is presented in
the next section as Theorem \ref{thm:rec}. Our main theorem, Theorem
\ref {thm:mainresultQb} above, can be regarded as a corollary
thereof. We present its proof here, using Theorem \ref{thm:rec}.

\begin{proof}[Proof of Theorem \ref {thm:mainresultQb}] We have to
  show a pointwise and a weak convergence.

  \smallskip {\em Step 1: Pointwise convergence of regularized
    profiles.}  We set $\hat V^\rho := (J_b\circ \mathcal{R}_\rho) (\hat
  u_0)$ and apply Theorem \ref{thm:rec} to these regularized profile
  evolutions $\hat V^\rho = \hat V^\rho(\xi,q,\tau)$. The function
  $V^\rho := \mathcal{F}_1^{-1} \hat V^\rho$ satisfies Assumption
  \ref{ass:growthcondition} by the assumptions of Theorem \ref
  {thm:mainresultQb}.  Theorem \ref{thm:rec} provides, for fixed
  $k\neq 0$,
  \begin{equation}\label{eq:cor-of-33}
    (\mathcal{F}_d\circ\mathcal{S})(V^\rho)(k,\tau/\eps^2)e^{i c |k| \tau/\eps^2}
    \to \left(\frac{|k|}{2\pi i}\right)^{-(d-1)/2}
    \hat V^\rho\left(|k|,\frac{k}{|k|},\tau\right)
  \end{equation}
  as $\eps\to 0$. It remains to calculate the two sides of this relation.

  The term $(\mathcal{F}_d\circ\mathcal{S})(V^\rho)$ on the left hand side of \eqref {eq:cor-of-33} is
  $(\mathcal{F}_d\circ\mathcal{S})(V^\rho) =
  (\mathcal{F}_d\circ\mathcal{S}\circ \mathcal{F}_1^{-1} \circ J_b\circ
  \mathcal{R}_\rho) (\hat u_0) = \hat{\mathcal{Q}}^\rho_b\hat u_0$. We
  see that the left hand side in \eqref {eq:cor-of-33} coincides with the left hand side in \eqref
  {eq:thm-ptw}.

  For $k\neq 0$, we calculate for the right hand side of \eqref
  {eq:cor-of-33}, using $\hat V_0^\rho := \mathcal{R}_\rho \hat u_0$,
  \begin{align*}
    &\left(\frac{|k|}{2\pi i}\right)^{-(d-1)/2}
    \hat V^\rho\left(\xi=|k|,q=\frac{k}{|k|},\tau\right)\\
    &\qquad = \left(\frac{|k|}{2\pi i}\right)^{-(d-1)/2}
    e^{-ib(|k|)\tau} \hat V_0^\rho
    \left(\xi=|k|,q=\frac{k}{|k|}\right)\\
    & \qquad =
    \begin{cases}
      e^{-ib(|k|)\tau}|k|^{-(d-1)/2} W_\rho(|k|)\hat u_0(k)
      \quad & \text{for }|k|< \rho\\
      e^{-ib(|k|)\tau}\hat u_0(k)
      \quad & \text{for }|k|\geq \rho\,.
    \end{cases}
  \end{align*}
  We have used that $W_\rho(|k|) = |k|^{(d-1)/2}$ for $|k|>\rho$. For
  $|k|\ge\rho$ the right hand side of \eqref {eq:cor-of-33} coincides
  with the right hand side of \eqref {eq:thm-ptw}. This provides the
  pointwise convergence.

  \smallskip {\em Step 2: Weak convergence.}  The operators
  $\hat{\mathcal{Q}}_b$ are bounded, uniformly in $\eps>0$. Therefore,
  the left hand side of \eqref {eq:thm-weak} is bounded in
  $L^2(\R^d)$, for every $\tau>0$. Upon choosing a subsequence
  $\eps\to 0$, for some limit function $L_\tau : \R^d\to \C$, $L_\tau
  = L_\tau(k)$, we can assume
  \begin{equation}
    \label{eq:some-limit}
    (\hat{\mathcal{Q}}_b \hat u_0)(\cdot,\tau/\eps^2) 
    e^{ic|\cdot|\tau/\eps^2} \weakto L_\tau
  \end{equation}
  weakly in $L^2(\R^d)$.  It remains to identify the limit $L_\tau$ as $e^{-ib(|\cdot|)\tau}\hat u_0$.
  The pointwise convergence of Step 1 implies that, for every $\rho>0$
  and every $\tau>0$,
  \begin{equation}
    \label{eq:weakconvrestricted}
    \left[ \hat{\mathcal{Q}}^\rho_b  \hat u_0(k,\tau/\eps^2)e^{ic|k|\tau/\eps^2} -  e^{-ib(|k|)\tau}\hat u_0(k)\right]\,
    \one_{\{|k|\geq \rho\}} \weakto 0\,.
  \end{equation}

  \smallskip {\em Identification of $L_\tau$.}  Let $f\in
  C_c^\infty(\R^d)$ be a smooth test function. We calculate
  \begin{align*}
    &\int_{\R^d}\left[(\hat{\mathcal{Q}}_b \hat u_0)(k,\tau/\eps^2)e^{ic|k|\tau/\eps^2}-e^{-ib(|k|)\tau}\hat u_0(k)\right]f(k)\,dk\\
    &\quad =\int_{\R^d}\left(\left(\hat{\mathcal{Q}}_b-\hat{\mathcal{Q}}^\rho_b\right)
      \hat u_0\right)(k,\tau/\eps^2)e^{ic|k|\tau/\eps^2}f(k)\,dk\displaybreak[2]\\
    &\qquad +\int_{\R^d}(\hat{\mathcal{Q}}^\rho_b\hat u_0)(k,\tau/\eps^2)
    (1-\one_{\{|k|\geq\rho\}})e^{ic|k|\tau/\eps^2}f(k)\,dk\\
    &\qquad + \int_{\R^d}\left[\left(\hat{\mathcal{Q}}^\rho_b\hat u_0\right)(k,\tau/\eps^2)
      e^{ic|k|\tau/\eps^2}-e^{-ib(|k|)\tau}\hat u_0(k)\right]\one_{\{|k|\geq \rho\}}f(k)\,dk\\
    &\qquad +\int_{\R^d}e^{-ib(|k|)\tau}\hat u_0(k)(\one_{\{|k|\geq\rho\}}-1)f(k)\,dk 
    \displaybreak[2]\\
    &\quad =:I_{\eps,\rho} + II_{\eps,\rho} + III_{\eps,\rho} +
    IV_{\eps,\rho}\,.
  \end{align*}

  Regarding the error term $I_{\eps,\rho}$, we use the fact that the
  operator $\mathcal{S}\circ \mathcal{F}_1^{-1}\circ J_b$ is bounded:
  \begin{align*}
    &\left\|\left(\left(\hat{\mathcal{Q}}_b-\hat{\mathcal{Q}}^\rho_b\right)
        \hat u_0\right)(\cdot,\tau/\eps^2)e^{ic|\cdot|\tau/\eps^2}\right\|_{L^2(\R^d)}\\
    &\qquad =\|\left(\left(\mathcal{S}\circ \mathcal{F}_1^{-1}\circ J_b
        \circ(\mathcal{R}-\mathcal{R_\rho})\right)\hat u_0\right)(\cdot,\tau/\eps^2)\|_{L^2(\R^d)}\\
    &\qquad \leq C\|(\mathcal{R}-\mathcal{R_\rho})\hat u_0\|_{X_S}\leq
    \tilde C\|W-W_\rho\|_{L^2(\R)}\to 0
  \end{align*}
  as $\rho\to 0$.  In the last step we have used that $\hat u_0$ is
  bounded. This allows to calculate $I_{\eps,\rho}$ in the limit $\rho\to 0$,
  \begin{align*}
    |I_{\eps,\rho}|
    &\leq \left\|\left(\left(\hat{\mathcal{Q}}_b-\hat{\mathcal{Q}}^\rho_b\right)
        \hat u_0\right)(\cdot,\tau/\eps^2)e^{ic|\cdot|\tau/\eps^2}\right\|_{L^2(\R^d)}
    \|f\|_{L^2(\R^d)}\\
    &\leq \tilde C\|W-W_\rho\|_{L^2(\R)}\|f\|_{L^2(\R^d)}\to 0\,.
  \end{align*}

  For the second error term we calculate
  \begin{align*}
    |II_{\eps,\rho}|
    &\leq\int_{\R^d}\left|(\hat{\mathcal{Q}}^\rho_b\hat u_0)(k,\tau/\eps^2)\right|\,
    \one_{\{|k|<\rho\}}|f(k)|\,dk\\
    &\leq \|f\|_\infty\|(\hat{\mathcal{Q}}^\rho_b\hat u_0)(\cdot,\tau/\eps^2)\|_{L^2(\R^d)}\,
    |\{|k|<\rho\}|^{1/2}\\
    &\leq C\|f\|_\infty\|\hat u_0\|_{L^2(\R^d)}\rho^{d/2}\,,
  \end{align*}
  where we have used that the linear operators
  $\hat{\mathcal{Q}}^\rho_b$ are bounded, independent of $\rho$.  

  For the third error term $III_{\eps,\rho}$ we exploit, for $\rho>0$
  fixed, the weak convergence \eqref{eq:weakconvrestricted}.  Finally,
  $IV_{\eps,\rho}$ is estimated by
  \begin{align*}
    |IV_{\eps,\rho}|
    \leq \int_{\R^d}|\hat u_0(k)|\one_{\{|k|<\rho\}}|f(k)|\,dk
    \leq C \|\hat u_0\|_{L^2(\R^d)}\, \rho^{d/2}\|f\|_\infty\,.
  \end{align*}

  In order to conclude the identification of the weak limit $L_\tau$,
  we first choose $\rho>0$ small such that
  $I_{\eps,\rho},II_{\eps,\rho}$ and $IV_{\eps,\rho}$ are
  small. Afterwards, we choose $\eps>0$ to achieve smallness in
  $III_{\eps,\rho}$. We find
  \begin{equation*}
    \int_{\R^d}(\hat{\mathcal{Q}}_b \hat u_0)(k,\tau/\eps^2)
    e^{ic|k|\tau/\eps^2}f(k)\,dk\to \int_{\R^d}e^{-ib(|k|)\tau}\hat u_0(k)f(k)\,dk
  \end{equation*}
  as $\eps\to 0$. Since $f\in C_c^\infty(\R^d)$ was arbitrary, we
  conclude
  \begin{equation*}
    L_\tau(k) = e^{-ib(|k|)\tau} \hat u_0(k)\,.
  \end{equation*}
  This shows \eqref {eq:thm-weak} and concludes the proof.
\end{proof}

\paragraph{Interpretation.}

Two choices of the function $b$ are of particular interest.

\smallskip {\bf 1)} $b(\xi)=0$ for all $\xi\in \R$. We recall that, by
our assumption on the initial data, the solution of the linear wave
equation is given in Fourier space by
\begin{equation}
  \label{eq:u-eps-Fourier-1}
  \hat u(k,\tau/\eps^2) = e^{- i c |k| \tau/\eps^2} \hat u_0(k)\,.
\end{equation}
Theorem \ref{thm:mainresultQb} implies that, in the limit $\eps\to 0$,
the solution $\hat u$ is close to the function $\hat{\mathcal{Q}}_b
\hat u_0$. This means that the ring solution with profile function $V
= \mathcal{F}_1^{-1}\circ J_b\circ\mathcal{R}\hat u_0$ is a good
approximation of $u$. The pointwise convergence \eqref {eq:thm-ptw}
implies Theorem \ref {thm:simple-version}.

\smallskip {\bf 2)} $b(\xi)=b_3\xi^3$ for all $\xi\in \R$. The weakly
dispersive equation
\begin{equation}
  \label{eq:weaklydispersive}
  \partial^2_t u(x,t)-c^2\Delta_x u(x,t)+ \eps^2d_0\Delta^2_x u(x,t) = 0
\end{equation}
with $d_0>0$ is an effective model to describe waves in heterogeneous
media or in discrete media, see \cite {DLS14} and \cite
{Schweizer-Theil-2017}. The Fourier transform of $u$ satisfies
\begin{equation}
  \partial_t^2\hat u(k,t) + c^2|k|^2\hat u(k,t) + \eps^2 d_0|k|^4\hat u(k,t)=0\,.
\end{equation}
With appropriate initial data, the solution to \eqref
{eq:weaklydispersive} is given in Fourier space by
\begin{equation}
  \label{eq:u-eps-Fourier-2}
  \hat u(k,t) = e^{- i \sqrt{c^2|k|^2+\eps^2d_0|k|^4}\, t} \hat u_0(k)\,.
\end{equation}
Expanding the square route in a Taylor series and considering large
times $t=\tau/\eps^2$ we find that
\begin{align*}
  \sqrt{c^2|k|^2+\eps^2d_0|k|^4}\tau/\eps^2
  &=\left(\sqrt{c^2|k|^2} + \frac{\eps^2d_0|k|^4}{2\sqrt{c^2|k|^2}}\right)\tau/\eps^2 
  + O(\eps^2)\\
  &=c|k|\tau/\eps^2+ \frac{d_0|k|^3}{2c}\tau + O(\eps^2)\,.
\end{align*} 
We set $b_3:=\frac{d_0}{2c}$ and use Theorem \ref {thm:mainresultQb}.
We conclude that, in the limit $\eps\to 0$, the solution $u$ is well
approximated by $\mathcal{Q}_b \hat u_0$: The profile function $V =
\mathcal{F}_1^{-1} \circ J_b\circ\mathcal{R} \hat u_0$ provides a good
approximation of the solution of the weakly dispersive equation
\eqref{eq:weaklydispersive}. The profile $V$ satisfies the linearized
KdV-equation
\begin{equation*}
  \partial_\tau V(z,q,\tau)=b_3\partial^3_z V(z,q,\tau)\,.
\end{equation*}
With this result we recover the profile analysis of \cite
{Schweizer-Theil-2017} in dimension $d=1$ and $d=2$, and extend it to
dimension $d=3$.

\section{Analysis of the reconstruction operator}
\label{sec.reconstruction-analysis}

Our main result Theorem \ref{thm:mainresultQb} states that solutions
to a (dispersive) wave equation can be recovered approximately by the
reconstruction operator $\mathcal{Q}_b$. This requires a study of the
expression $\hat{\mathcal{Q}}_b \hat u_0 = \mathcal{F}_d\circ
\mathcal{S}\circ \mathcal{F}_1^{-1}\circ J_b\circ \mathcal{R} (\hat
u_0)$. As we have already seen, the core result regards the outer
part, the expression $(\mathcal{F}_d\circ \mathcal{S}) V$. This part
is analyzed in Theorem \ref{thm:rec} below.

\begin{assumption}
  \label{ass:growthcondition}
  Let the dimension be $d\in\{1,2,3\}$.  On $V\in
  L^\infty(0,\infty;X_S)$ we assume the following.
  \begin{itemize}
  \item[(i)] There exist $C,\alpha>0$ such that for every $\tau\in
    (0,\infty)$ and $q\in S^{d-1}$
    \begin{equation}
      \label{eq:growthcondition1}
      |V(z,q,\tau)|\leq C(1+|z|)^{-d-\alpha}\,. 
    \end{equation}
  \item[(ii)] The Fourier transform $\hat V := \mathcal{F}_1 V$ has
    the property that, for every $\tau\in (0,\infty)$, the function
    \begin{equation*}
      \R\times S^{d-1} \ni (\xi, q) \mapsto \hat V(\xi ,q, \tau) \in \C
    \end{equation*}
    is of class $C^{d-1}(\R\times S^{d-1};\C)$.
  \end{itemize}
\end{assumption}

In Theorems \ref {thm:mainresultQb} and \ref {thm:rec}, we demand that
$V^\rho:=(\mathcal{F}_1^{-1}\circ J_b\circ \mathcal{R}_\rho)\hat u_0$
satisfies Assumption \ref{ass:growthcondition}. Actually, this is not
too restrictive.
\begin{remark}
  \label{rem:regularityassumptions}
  Let $\hat u_0$ be a smooth function with compact support. Then $\hat
  V^\rho := (J_b\circ \mathcal{R}_\rho)\hat u_0$ has also compact
  support. Moreover, since $\mathcal{R}_\rho$ uses the regularization
  of $|\xi|^{(d-1)/2}\one_{\{\xi>0\}}$, the smoothness of $\hat u_0$
  is inherited by $\hat V^\rho$.  Smoothness of $\hat V^\rho$ implies
  the decay property \eqref{eq:growthcondition1} of $V^\rho =
  \mathcal{F}_1^{-1} \hat V^\rho$ in $z$. We conclude that Assumption
  \ref{ass:growthcondition} is satisfied.
\end{remark}
 
We are now in the position to prove our core result. 

\begin{theorem} [The shell operator in Fourier space]
  \label{thm:rec}
  In dimension $d\in \{1,2,3\}$ let $\hat V\in L^\infty(0,\infty;X_s)$
  satisfy $\hat V(\xi,q,\tau) = 0$ for every $\xi<0$ and let $V :=
  \mathcal{F}_1^{-1} \hat V^\rho\in L^\infty(0,\infty;X_s)$ satisfy
  Assumption \ref{ass:growthcondition}. Consider the ring-like
  solution $\mathcal{S} V$ and its Fourier transform
  $\mathcal{F}_d\circ \calS(V)$. For every $k\in \R^d\setminus\{0\}$
  and every $\tau>0$ holds
  \begin{equation}
    \label{eq:claim-thm}
   (\mathcal{F}_d\circ \calS)(V)(k,\tau/\eps^2)e^{i c |k| \tau/\eps^2}
    \to \left(\frac{|k|}{2\pi i}\right)^{-(d-1)/2}
    \hat V\left(|k|,q =\frac{k}{|k|},\tau\right)
  \end{equation}
  as $\eps\to 0$. Moreover, the convergence holds as weak convergence
  in $L^2(\R^d)$.
\end{theorem}

\begin{proof}[Proof of Theorem \ref{thm:rec}] 
  It suffices to prove, for $k\neq 0$,
  \begin{equation}
    \label{eq:claim-thm2}
    Q^\eps(k,\tau) 
    := e^{i c |k| \tau/\eps^2}\, (\mathcal{F}_d\circ \calS)(V)(k,\tau/\eps^2)
    \to\left(\frac{|k|}{2\pi i}\right)^{-(d-1)/2}
    \hat V\left(|k|,q=\frac{k}{|k|},\tau\right)\,.
  \end{equation}
  Indeed, since the operator $\mathcal{F}_d\circ \calS$ is bounded and
  since $|e^{i c |k| \tau/\eps^2}|=1$, for every $\tau>0$, the
  sequence $Q^\eps(\cdot,\tau)$ is uniformly bounded in
  $L^2(\R^d)$. Therefore there exists, up to a subsequence, a weak
  limit in $L^2(\R^d)$. Since weak and pointwise limits always
  coincide, we conclude the weak convergence of $Q^\eps(\cdot,\tau)$
  to the right hand side of \eqref{eq:claim-thm2}.
		
  We show the pointwise convergence in five steps.

  \smallskip {\em Step 1: Calculation of the quantity of interest.}
  We calculate the left hand side of \eqref {eq:claim-thm2}.
  Definition \ref{def:shelloperator} of the shell operator
  $\mathcal{S}$ provides
  \begin{align*}
    \mathcal{S}V(x,t) =  
    \frac1{(ct)^{(d-1)/2}} V\left(|x| - ct, \frac{x}{|x|},\eps^2 t\right)\one_{\{|x|<2ct\}}\,.
  \end{align*}
  We calculate the Fourier transform in polar coordinates, $x = r q$
  with $r>0$ and $q\in S^{d-1}$,
  \begin{align*}
    (\mathcal{F}_d\circ \calS) (V)(k,t) &= \int_{\R^d} e^{-ix\cdot k}
    (\mathcal{S}V)(x,t)\, dx\\ 
    &= \int_0^{\infty} \int_{S^{d-1}} r^{d-1}
    e^{-ir q\cdot k} (\mathcal{S}V)(rq,t)\, dS(q)\, dr\,.
  \end{align*}
  We insert $\mathcal{S} V$ from above. Evaluating in $t =
  \tau/\eps^2$ we find
  \begin{align}
    \label{eq:formula-star}
    Q^\eps(k,\tau)
    = e^{i c |k| \tau/\eps^2}\int_0^{2c\tau/\eps^2} \int_{S^{d-1}}  
    \frac{r^{d-1} e^{- ir q\cdot k}}{(c\tau/\eps^2)^{(d-1)/2}} 
    V\left(r-c\frac{\tau}{\eps^2},q, \tau\right)
    dS(q)\, dr\,.
  \end{align}
  To simplify, we write $r = c\tau/\eps^2 + z$ with a new variable
  $z\in \R$; the integration over $r$ is replaced by an integration
  over $z$. We find
  \begin{equation*}
    Q^\eps(k,\tau)
     =e^{i c |k| \tau/\eps^2}\int_{-c\tau/\eps^2}^{c\tau/\eps^2} \int_{S^{d-1}} 
    \frac{(c\tau/\eps^2 + z)^{d-1}}{(c\tau/\eps^2)^{(d-1)/2}} e^{-i q\cdot k\, c\tau/\eps^2} 
    e^{-i z q\cdot k}\, 
     V(z,q,\tau)\,
    dS(q)\, dz\,.
    \label{eq:formula-star2}
 \end{equation*}

 \smallskip {\em Step 2: Approximation.}  We treat the cases $d \in
 \{1,2\}$ and $d=3$ differently.

 \smallskip {\em Case $d \in \{1,2\}$. }  We use the approximations
 $\int_{-c\tau/\eps^2}^{c\tau/\eps^2} \approx \int_\R$ and
 $\frac{(c\tau/\eps + z)^{d-1}}{(c^\tau/\eps^2)^{(d-1)/2}} \approx
 (c\tau/\eps^2)^{(d-1)/2}$ and write
  \begin{equation}
    \label{eq:Q-AFG-0}
    Q^\eps(k,\tau) = A_0^\eps(k,\tau) + G_0^\eps(k,\tau)
  \end{equation}
  with
  \begin{align*}
    &A_0^\eps(k,\tau) = e^{i c |k| \tau/\eps^2}\int_\R \int_{S^{d-1}} 
    (c\tau/\eps^2)^{(d-1)/2} e^{-i q\cdot k\, c\tau/\eps^2} 
    e^{-i z q\cdot k}\, 
     V(z,q,\tau)\,
    dS(q)\, dz\,,\\
    & G_0^\eps(k,\tau)= 
    e^{i c |k| \tau/\eps^2}\int_\R\int_{S^{d-1}}e^{-i q\cdot k( c\tau/\eps^2+z)}V(z,q,\tau) \times\\
    &\qquad\qquad\qquad\qquad \times \left[\frac{(c\tau/\eps^2 + z)^{d-1}}{(c\tau/\eps^2)^{(d-1)/2}}
      \one_{\{|z|<c\tau/\eps^2\}}-(c\tau/\eps^2)^{(d-1)/2}\right] \,dS(q)\, dz\,.
  \end{align*}

  {\em Case d=3:} In three dimensions, we use higher order
  approximations: $\int_{-c\tau/\eps^2}^{c\tau/\eps^2} \approx
  \int_\R$ and $\frac{(c\tau/\eps^2 + z)^{2}}{(c\tau/\eps^2)} \approx
  c\tau/\eps^2+2z$.  This allows to write
  \begin{equation}
    \label{eq:Q-AFG-1}
    Q^\eps(k,\tau) = A_1^\eps(k,\tau) + G_1^\eps(k,\tau)
  \end{equation}
  with
  \begin{align*}
    &A_1^\eps(k,\tau) = e^{i c |k| \tau/\eps^2}\int_\R \int_{S^{d-1}}
    (c\tau/\eps^2 + 2z) e^{-i q\cdot k\, c\tau/\eps^2} e^{-i z q\cdot
      k}\, V(z,q,\tau)\,
    dS(q)\, dz\,,\\
    & G_1^\eps(k,\tau)=
    e^{i c |k| \tau/\eps^2}\int_\R\int_{S^{d-1}}e^{-i q\cdot k( c\tau/\eps^2+z)}V(z,q,\tau)\times\\
    &\qquad\qquad\qquad\qquad \times \left[\frac{(c\tau/\eps^2 +
        z)^2}{(c\tau/\eps^2)}\one_{\{|z|<c\tau/\eps^2\}}-(c\tau/\eps^2+2z)\right]
    \,dS(q)\, dz\,.
  \end{align*}

  \smallskip {\em Step 3: Simplifying the expression for $A^\eps_0,
    A^\eps_1$.} One of the integrals in the formulas for $A_i^\eps$
  can be evaluated. Indeed, in $A_0^\eps$ and in one of the two terms
  of $A_1^\eps$, we recognize
  \begin{equation*}
    \int_\R  V(z,q,\tau)e^{-i z q\cdot k}\,dz = \hat V(q\cdot k,q,\tau)\,.
  \end{equation*}
  In dimension $d=3$, we find
  \begin{equation*}
    \int_\R  z\, V(z,q,\tau)e^{-i z q\cdot k}\,dz 
    = i\partial_{\xi}\hat V(\xi=q\cdot k,q,\tau)\,,
  \end{equation*}
  where integrability of all terms is assured by Assumption \ref
  {ass:growthcondition}.
	
  The formula for $A_0^\eps$ simplifies to
  \begin{align}
    \nonumber 
    &A^\eps_0(k,\tau) =
    \int_{S^{d-1}} (c\tau/\eps^2)^{(d-1)/2}\,
    e^{i (|k| - q\cdot k)\, c\tau/\eps^2}\,
    \hat V(q\cdot k,q,\tau)\, dS(q)\\
    &\ = \int_{S^{d-1}} \left(\frac{|k|}{2\pi i} c\tau/\eps^2\right)^{(d-1)/2}\,
    e^{i (1 - q\cdot k/|k|)\, |k| c\tau/\eps^2}\,
    \left[ \left(\frac{|k|}{2\pi i}\right)^{-(d-1)/2} \hat V(q\cdot k,q,\tau)\right]
    \, dS(q)\,.
    \label{eq:formula-star3} 
  \end{align}
  The formula for $A_1^\eps$ simplifies to
  \begin{align}
    \nonumber 
    &A^\eps_1(k,\tau) 
    = \int_{S^2}   e^{i (|k| - q\cdot k)\, c\tau/\eps^2}\,
    \left( c\tau/\eps^2\hat V(q\cdot k,q,\tau) 
      + 2i\partial_{\xi}\hat V(q\cdot k,q,\tau)\right)\, dS(q)\nonumber\\
    &\ = \int_{S^2} \left(\frac{|k|}{2\pi i} c\tau/\eps^2\right)\,
    e^{i (1 - q\cdot k/|k|)\, |k| c\tau/\eps^2}\,\left[\left(\frac{|k|}{2\pi i}\right)^{-1} 
      \hat V(q\cdot k,q,\tau)\right]\, dS(q)\nonumber\\
    &\quad + \frac{2\eps^2}{c\tau}i\int_{S^2} \left(\frac{|k|}{2\pi i} c\tau/\eps^2\right)\,
    e^{i (1 - q\cdot k/|k|)\, |k| c\tau/\eps^2}\,\left[\left(\frac{|k|}{2\pi i}\right)^{-1}
      \partial_{\xi} \hat V(q\cdot k,q,\tau)\right]\, dS(q)\,.
    \label{eq:formula-star4} 
  \end{align}

  \smallskip {\em Step 4: Application of a stationary phase limit.}
  We consider the terms in squared brackets in \eqref
  {eq:formula-star3} and \eqref {eq:formula-star4} as
  test-functions. Denoting them as $\phi = \phi(q)$, we exploit Lemma
  \ref{lem:spherical} to calculate the limit $\eps\to 0$. The lemma
  provides
  \begin{align}
    \label{eq:integraltestfunction}
    \int_{S^{d-1}} \left(\frac{|k|}{2\pi i} c\tau/\eps^2\right)^{(d-1)/2}\,
    e^{i (1 - q\cdot k/|k|)\, |k| c\tau/\eps^2}\, \phi(q) \, dS(q)\to \phi(k/|k|)\,.
  \end{align}
  Indeed, since $k\in \R^d\setminus\{0\}$ is held fixed, we can use
  Lemma \ref{lem:spherical} with $\kappa := k/|k|$ and the sequence of
  numbers $N := |k|c\tau/\eps^2$, which tends to $+\infty$.

  Let us check if the assumptions of Lemma \ref{lem:spherical} are
  satisfied. The lemma requires that $\phi:S^{d-1} \to \C$ is
  supported on the half sphere defined by $\kappa$. This requirement
  is satisfied since we demanded $\hat V(\xi,q,\tau) = 0$ for every
  $\xi<0$.  Moreover, Lemma \ref{lem:spherical} requires that
  $\phi:S^{d-1} \to \C$ is of class $C^1$. In dimension $d=1$, this is
  no further requirement. In dimension $d=2$, we need that $q\mapsto
  \hat V(q\cdot k,q,\tau)$ is of class $C^1$; this follows from
  Assumption \ref{ass:growthcondition}, (ii). In dimension $d=3$, we
  need that both $q\mapsto \hat V(q\cdot k,q,\tau)$ and $q\mapsto
  \del_\xi \hat V(q\cdot k,q,\tau)$ are of class $C^1$; also this
  follows from Assumption \ref{ass:growthcondition}, (ii).

  \smallskip The second term in \eqref{eq:formula-star4} vanishes in
  the limit as $\eps\to 0$ due to \eqref{eq:integraltestfunction} and
  the factor $\eps^2$ in front of the integral.  The limits of the
  remaining terms are determined by evaluating $\hat V(q\cdot
  k,q,\tau)$ in the point $q = \kappa = k/|k|$. We find $\hat
  V\left(\frac{k}{|k|}\cdot k,\frac{k}{|k|},\tau\right) = \hat
  V\left(|k|,\frac{k}{|k|},\tau\right)$. This yields, for $k\neq 0$,
  \begin{equation*}
    \lim_{\eps\to 0} A^\eps_0(k,\tau)
    = \lim_{\eps\to 0} A^\eps_1(k,\tau)
    = \left(\frac{|k|}{2\pi i}\right)^{-(d-1)/2} 
    \hat V\left(|k|,\frac{k}{|k|},\tau\right)\,. 
  \end{equation*}
  This is the desired limit in \eqref {eq:claim-thm2}. Once we show
  that the error terms $G_0^\eps$ and $G^\eps_1$ are small, we have
  shown \eqref {eq:claim-thm2} and hence the Theorem.

  \smallskip {\em Step 5: Calculation of the error terms $G_0^\eps$
    and $G^\eps_1$.}  We show the result for the three dimensions
  separately.
	
  \smallskip {\em Dimension $d=1$.} In the case $d=1$ we have
  \begin{equation*}
    G_0^\eps(k,\tau)
    =e^{i c |k|\tau/\eps^2} \sum_{q=\pm 1}\int_{\R} e^{-i q\cdot k( c\tau/\eps^2+z)} 
    V(z,q,\tau)  \one_{\{|z|\geq c\tau/\eps^2\}}\, dz\,.
  \end{equation*}
  Exploiting $\left|e^{i c |k|\tau/\eps^2}e^{-i q\cdot k(
      c\tau/\eps^2+z)}\right| = 1$ we find
  \begin{align*}
    |G_0^\eps(k,\tau)|\leq \sum_{q=\pm 1}
    \int_{\R}|V(z,q,\tau)|\, \one_{\{|z|\geq c\tau/\eps^2\}}\, dz \to 0
  \end{align*}
  as $\eps\to 0$; here we exploit that Assumption
  \ref{ass:growthcondition} provides a decay rate that assures
  $V(\cdot, q,\tau)\in L^1(\R)$ uniformly in $q$ and $\tau$.
	
  \smallskip {\em Dimension $d=2$.} In the case $d=2$ we find
  \begin{align*}	
    |G_0^\eps(k,\tau)|\leq \int_\R\int_{S^{1}}|V(z,q,\tau)|
    \left|\frac{c\tau/\eps^2 + z}{(c\tau/\eps^2)^{1/2}}
      \one_{\{|z|<c\tau/\eps^2\}}-(c\tau/\eps^2)^{1/2}\right| \,dS(q)\, dz\,.
  \end{align*}
  Since $S^1$ has the finite measure $2\pi$, it suffices to show the
  convergence
  \begin{equation*}
    \int_\R|V(z,q,\tau)|
    \left|\frac{c\tau/\eps^2 + z}{(c\tau/\eps^2)^{1/2}}
      \one_{\{|z|<c\tau/\eps^2\}}-(c\tau/\eps^2)^{1/2}\right|dz\to  0
  \end{equation*}
  as $\eps\to 0$, uniformly in $q\in S^1$. We decompose the integral
  into two parts, $|z|\leq \delta/\eps$ and $|z|>\delta/\eps$ with
  $\delta>0$ to be chosen below.  We only consider $\eps$-values with
  $c\tau/\eps>\delta$, such that
  \begin{align*}
    &\int_{|z|\leq \delta/\eps}|V(z,q,\tau)|
    \left|\frac{c\tau/\eps^2 + z}{(c\tau/\eps^2)^{1/2}}
      \one_{\{|z|<c\tau/\eps^2\}}-(c\tau/\eps^2)^{1/2}\right|dz\\
    &\quad = \int_{|z|\leq \delta/\eps}|V(z,q,\tau)|
    \left|\frac{c\tau/\eps^2 + z}{(c\tau/\eps^2)^{1/2}}-(c\tau/\eps^2)^{1/2}\right|dz\,.
  \end{align*}
  Using
  \begin{equation}
    \label{eq:estimatepowers}
    \left|\frac{c\tau/\eps^2 + z}{(c\tau/\eps^2)^{1/2}}-(c\tau/\eps^2)^{1/2}\right|
    =\left|\frac{z}{(c\tau/\eps^2)^{1/2}}\right|
    =\eps\frac{z}{(c\tau)^{1/2}}\leq \frac{\delta}{(c\tau)^{1/2}}
  \end{equation}
  for $|z|\leq \delta/\eps$, we obtain
  \begin{align*}
    &\int_{|z|\leq \delta/\eps}|V(z,q,\tau)|
    \left|\frac{c\tau/\eps^2 + z}{(c\tau/\eps^2)^{1/2}}-(c\tau/\eps^2)^{1/2}\right|dz
    \leq \frac{\delta}{(c\tau)^{1/2}}\int_{|z|\leq \delta/\eps}|V(z,q,\tau)|\,dz\\
    &\quad \leq  \frac{\delta}{(c\tau)^{1/2}}\int_{\R}|V(z,q,\tau)|\,dz\leq C\delta
  \end{align*}
  with $C = C(\tau)$, where we have used that $V(\cdot, q,\tau)\in
  L^1(\R)$ uniformly in $q$ and $\tau$. The integral over
  $|z|>\delta/\eps$ is estimated exploiting
  \begin{align*}
    &\left|\frac{c\tau/\eps^2 + z}{(c\tau/\eps^2)^{1/2}}
      \one_{\{|z|<c\tau/\eps^2\}}-(c\tau/\eps^2)^{1/2}\right|
    \leq \left|\frac{c\tau/\eps^2 + z}{(c\tau/\eps^2)^{1/2}}\right| +(c\tau/\eps^2)^{1/2} \\
    &\quad \leq 
    2(c\tau/\eps^2)^{1/2} + \left|\frac{z}{(c\tau/\eps^2)^{1/2}}\right|
    =\frac{2}{\eps}(c\tau)^{1/2} 
    + \eps\frac{|z|}{(c\tau)^{1/2}}\,.
  \end{align*}
  We find
  \begin{align*}
    &\int_{|z|> \delta/\eps}|V(z,q,\tau)|
    \left|\frac{c\tau/\eps^2 + z}{(c\tau/\eps^2)^{1/2}}
      \one_{\{|z|<c\tau/\eps^2\}}-(c\tau/\eps^2)^{1/2}\right|\,dz\displaybreak[2]\\
    &\quad \leq \int_{|z|> \delta/\eps}|V(z,q,\tau)|\left(\frac{2}{\eps}(c\tau)^{1/2} 
      + \eps\frac{|z|}{(c\tau)^{1/2}}\right)\,dz\displaybreak[2]\\
    &\quad \leq C\int_{|z|> \delta/\eps}|z|^{-2-\alpha}\left(\frac{2}{\eps}(c\tau)^{1/2} 
      + \eps\frac{|z|}{(c\tau)^{1/2}}\right)\,dz\\
    &\quad \leq C\left(\eps^{-1}(\eps/\delta)^{1+\alpha} + \eps(\eps/\delta)^\alpha\right)\,.
  \end{align*}
  In the last step we have exploited the assumption on $V$, namely
  $|V(z,q,\tau)|\leq C(1+|z|)^{-2-\alpha}$.  Choosing first $\delta>0$
  to have smallness in the first integral and then $\eps>0$ small to
  make the second integral small, we conclude
  \begin{equation*}
    |G^\eps_0(k,\tau)|\to 0\quad\text{for }\eps\to 0\,. 
  \end{equation*}
		
  {\em Dimension $d=3$.} The case $d=3$ is analogous to the case
  $d=2$. For the integral over $|z|\leq \delta/\eps$ we use
  \begin{align*}
    &\frac{(c\tau/\eps^2 + z)^{2}}{(c\tau/\eps^2)}-(c\tau/\eps^2 + 2z)
    =\frac{1}{\eps^2}\left(\frac{(c\tau+\eps^2z)^2}{c\tau}-c\tau\right)-2z\\
    &\quad = \frac{1}{\eps^2c\tau}\left((c\tau+\eps^2z)^2-(c\tau)^2\right)-2z
    =\frac{1}{\eps^2c\tau}\left(2c\tau\eps^2z+\eps^4z^2\right)-2z
    =\frac{\eps^2z^2}{c\tau}\leq \frac{\delta^2}{c\tau}
  \end{align*}
  and the fact that $V(\cdot, q,\tau)\in L^1(\R)$ uniformly in $q$.
  Concerning the integral over $|z|>\delta/\eps$ we calculate
  \begin{align*}
    \left|\frac{(c\tau/\eps^2 + z)^{2}}{(c\tau/\eps^2)}\one_{\{|z|<c\tau/\eps^2\}}-(c\tau/\eps^2+2z)\right|
    \leq 3c\tau/\eps^2  + \frac{2|z|^2}{c\tau/\eps^2} + 2|z|
  \end{align*}
  and thus
  \begin{align*}
    &\int_{|z|> \delta/\eps}|V(z,q,\tau)|
    \left|\frac{c\tau/\eps^2 + z}{(c\tau/\eps^2)^{1/2}}\one_{\{|z|<c\tau/\eps^2\}}-(c\tau/\eps^2)^{1/2}\right|\,dz\\
    &\quad \le \int_{|z|> \delta/\eps}|V(z,q,\tau)|\left(3c\tau/\eps^2  + \frac{2|z|^2}{c\tau/\eps^2} + 2|z|\right)\,dz\\
    &\quad \le C\int_{|z|> \delta/\eps}|z|^{-3-\alpha}\left(3c\tau/\eps^2  + \frac{2|z|^2}{c\tau/\eps^2} + 2|z|\right)\,dz\\
    &\quad \le C\left(\eps^{-2}(\eps/\delta)^{2+\alpha} + \eps^2(\eps/\delta)^\alpha + (\eps/\delta)^{1+\alpha}\right).
  \end{align*}
  In the last step we have exploited Assumption
  \ref{ass:growthcondition} on $V$, namely $|V(z,q,\tau)|\leq
  C(1+|z|)^{-3-\alpha}$ uniformly in $q$ and $\tau$.

  Once more, we choose first $\delta>0$ small to have the integral
  over $|z| \le \delta/\eps$ small. We then choose $\eps>0$ small to
  have the other integral small. We obtain that the error terms
  $G_0^\eps$ and $G^\eps_1$ vanish in the limit $\eps\to 0$. Up to the
  claim in \eqref {eq:integraltestfunction}, where we used the
  subsequent Lemma \ref {lem:spherical}, the theorem is shown.
\end{proof}

\section{A stationary phase convergence result }
\label{sec.stat-phase}

In the last section, the relevant small parameter was $\eps>0$; in
this section, we work with the large parameter $N := |k|c\tau/\eps^2$.
We applied in Section \ref {sec.reconstruction-analysis} the
subsequent Lemma \ref {lem:spherical} with the vector $\kappa :=
k/|k|$.

In the following, for arbitrary dimension $d\in\{1,2,3\}$, we will
demand that the test-function $\phi:S^{d-1}\to \R$ is of class
$C^1(S^{d-1})$ and that it is supported on the half-sphere $\{q\in
S^{d-1}\,|\,q\cdot \kappa \geq 0\}$.

Regarding the case $d=1$ we note that $S^{d-1} = \{ +1, -1\}$ and
that, for $\kappa = e_1 \equiv 1$, a function $\phi \in C^1(S^{d-1})$
with support in the half-sphere $\{q\in S^{d-1}\,|\, q\cdot \kappa
\geq 0\} = \{ 1\}$ is a function $\phi : \{ +1, -1\} \to \R$ with
$\phi(-1) = 0$.

\begin{lemma}\label {lem:spherical}
  Let the dimension be $d\in \{1,2,3\}$. Let $\kappa\in S^{d-1}$ be a
  point on the sphere and let $\phi \in C^1(S^{d-1};\R)$ be supported
  in $\{q\in S^{d-1}\,|\,q\cdot \kappa \geq 0\}$. Then there holds
  \begin{equation}
    \label{eq:Dirac-aim1}
    A^N_\phi := (2\pi i)^{-(d-1)/2} \int_{S^{d-1}} N^{(d-1)/2}\,
    e^{i (1 - q\cdot \kappa)\, N}\, \phi(q) \, dS(q)\to \phi(\kappa)
  \end{equation} 
  as $N\to \infty$.
\end{lemma}

\begin{proof}
  By radial symmetry it is sufficient to consider the case
  $\kappa:=e_1$.  We show the result for the three dimensions
  separately.

  \smallskip {\em Step 1: Dimension $d=1$.} In the case $d=1$, the
  integral in \eqref {eq:Dirac-aim1} is a sum of two terms, 
  \begin{equation}
    \label{eq:Dirac-aim1-d1}
    A^N_\phi = \sum_{q\in \{\pm 1\}} e^{i (1 - q\cdot 1)\, N}\, \phi(q) 
    = \phi(1) + e^{2 i\, N}\, \phi(-1) = \phi(1)\,.
  \end{equation}
  This shows \eqref {eq:Dirac-aim1}.

  \smallskip {\em Step 2: Dimension $d=3$.} 
  We use spherical coordinates
  \begin{align*}
    q(\theta,\vth) :=
    \dvect{\cos(\theta)}{\sin(\theta)\cos(\vth)}{\sin(\theta)\sin(\vth)}
  \end{align*}
  with angles $\theta\in (0,\pi)$ and $\vth\in(0,2\pi)$ and surface
  element $J := \sqrt{\det(Dq^T Dq)} = \sin(\theta)$.  We calculate
  the expression of \eqref{eq:Dirac-aim1} for $d=3$ with spherical
  coordinates as
  \begin{align*}
    A^N_\phi &= (2\pi i)^{-1} \int_{S^2} N\, e^{i (1 - q\cdot e_1)\, N}\, \phi(q) \, dS(q)\\
    &= (2\pi i)^{-1} \int_0^\pi \int_0^{2\pi} N\, e^{i (1 - \cos(\theta))\, N}\, 
    \phi(q(\theta,\vth))\, d\vth\, \sin(\theta)\, d\theta\\
    &=-\int_{0}^\pi i N\, e^{i (1 - \cos(\theta))\, N}\sin(\theta)
    \underbrace{\left(\frac{1}{2\pi}\int_0^{2\pi}
        \phi(q(\theta,\vth))\, d\vth\right)}_{=: \tilde\phi(\theta)}\,d\theta\displaybreak[2]\\
    &= - \int_0^\pi \frac{d}{d\theta} \left[e^{i (1 - \cos(\theta))\, N}\right]
    \tilde\phi(\theta)\, d\theta\\
    &= -\left[e^{i (1 - \cos(\theta))\, N} \tilde\phi(\theta)\right]^\pi_{\theta=0} +
    \int_0^\pi e^{i (1 - \cos(\theta))\, N}
   \frac{d}{d\theta} \tilde\phi(\theta)\, d\theta\\
	&= \tilde\phi(0) +\int_0^{\pi/2} e^{i (1 - \cos(\theta))\, N}
   \frac{d}{d\theta} \tilde\phi(\theta)\, d\theta\,,
  \end{align*}
  where integration by parts is allowed because of $\tilde \phi\in
  C^1([0,\pi])$. In the last line we also exploited
  $\tilde\phi(\theta)=0$ for $\theta\in(\pi/2, \pi)$.  For
  $\tilde\phi(0)$ we obtain
  \begin{equation*}
    \tilde\phi(0) 
    = \frac{1}{2\pi}\int_0^{2\pi}\phi(q(0,\vth))\, d\vth
    = \frac{1}{2\pi}\int_0^{2\pi}\phi(e_1)\,d\vth
    =\phi(e_1)\,.
  \end{equation*}

  We turn now to the treatment of the integral. We use the
  substitution $z=1-\cos(\theta)$ with $\frac{dz}{d\theta} =
  \sin(\theta) = \sqrt{1-\cos^2(\theta)} = \sqrt{1-(1-z)^2} =
  \sqrt{z}\sqrt{2-z}$ to obtain
  \begin{align*}
    \int_0^{\pi/2} e^{i (1 - \cos(\theta))\, N} 
    \frac{d}{d\theta} \tilde\phi(\theta)\, d\theta
    =\int_0^1e^{izN} \frac{d}{d\theta} \tilde\phi
    (\arccos(1-z))\frac{1}{\sqrt{z}\sqrt{2-z}}\,dz\,.
  \end{align*}
  The factor $z\mapsto e^{izN}$ is a sequence of highly oscillatory
  functions; it converges to the mean value
  $\frac{1}{2\pi}\int_0^{2\pi} e^{iy}\,dy = 0$ weakly in $L^p(0,1)$
  for every $p\in(1,\infty)$.  Since $\frac{d}{d\theta} \tilde\phi$ is
  bounded and $\frac{1}{\sqrt{2-z}}\leq 1$ for $z\in(0,1)$, we find that
  \begin{equation*}
    z\mapsto \frac{d}{d\theta} \tilde\phi(\arccos(1-z))\frac{1}{\sqrt{z}\sqrt{2-z}}
  \end{equation*}
  is in $L^q(0,1)$ for $q\in (1,2)$; it is thus an admissible test
  function for the weak convergence property.  We obtain
  \begin{equation*}
    \int_0^{\pi/2} e^{i (1 - \cos(\theta))\, N}
    \frac{d}{d\theta} \tilde\phi(\theta)\, d\theta\to 0 
  \end{equation*}
  as $N\to\infty$, which provides the claim \eqref {eq:Dirac-aim1} for
  $d=3$.

  \smallskip {\em Step 3: Dimension $d=2$.}  We use the coordinates
  $q(\theta) := (\cos(\theta), \sin(\theta))$ with
  $\theta\in(-\pi,\pi)$, the line element is $J = 1$.  The expression
  of \eqref {eq:Dirac-aim1} is
  \begin{align*}
    A^N_\phi &= (2\pi i)^{-1/2} \int_{S^1} N^{1/2}\, e^{i (1 - q\cdot e_1)\, N}\, 
    \phi(q) \, dS(q)\\
    &= (2\pi i)^{-1/2}\, \int_{-\pi}^{\pi} N^{1/2}\, e^{i (1 - \cos(\theta))\, N}\,
    \phi(q(\theta))\, d\theta\\
    &= (2\pi i)^{-1/2}\, \int_{0}^{\pi/2} N^{1/2}\, e^{i (1 - \cos(\theta))\, N}\,
    \tilde\phi(\theta)\, d\theta\,,
  \end{align*} 
  where $\tilde\phi(\theta):=\phi(q(\theta)) + \phi(q(-\theta))$
  denotes a symmetrized version of $\phi$. We split the integral into
  two parts, $\theta\in(0,\delta)$ and $\theta\in(\delta,\pi)$, where
  the small parameter $\delta$ is chosen $N$-dependent, $\delta :=
  N^{-\beta}$ with $\beta=3/10$.  We calculate
  \begin{align*}
    &\int_{N^{-3/10}}^{\pi/2} N^{1/2}\, e^{i (1 - \cos(\theta))\, N}\,
    \tilde\phi(\theta)\, d\theta\\
    &=\frac{1}{\sqrt{N}} \int_{N^{-3/10}}^{\pi/2} \sin(\theta) i\,N\, e^{i (1 - \cos(\theta))\, N}\,
    \frac{\tilde\phi(\theta)}{i\sin(\theta)}\, d\theta\\
    &=\frac{1}{\sqrt{N}}\left[e^{i (1 - \cos(\theta))\, N}
      \frac{\tilde\phi(\theta)}{i\sin(\theta)}\right]^{\pi/2}_{\theta=N^{-3/10}}
    -\frac{1}{\sqrt{N}}\int_{N^{-3/10}}^{\pi/2}e^{i (1 - \cos(\theta))\, N}\,
    \frac{\frac{d}{d\theta}\tilde\phi(\theta)}{i\sin(\theta)}\, d\theta\displaybreak[2]\\
    &\quad +\frac{1}{\sqrt{N}}\int_{N^{-3/10}}^{\pi/2}e^{i (1 - \cos(\theta))\, N}\,
    \frac{\tilde\phi(\theta)\cos(\theta)}{i\sin^2(\theta)}\, d\theta\displaybreak[2]\\
    &=: I_N + II_N + III_N \,.
  \end{align*}
  The terms $I_N, II_N$ vanish in the limit as $N\to\infty$. Indeed,
  for $N$ sufficiently large
  \begin{align*}
    |I_N| = 
    \frac{1}{\sin(N^{-3/10})\sqrt{N}}|\tilde\phi(N^{-3/10})|\leq C\frac{N^{3/10}}{\sqrt{N}}
    =CN^{3/10-1/2}\stackrel{N\to\infty}{\rightarrow} 0
  \end{align*}
  and, since sine is monotonically increasing in $(0,\pi/2)$,
  \begin{align*}
    |II_N|&\leq\frac{1}{\sqrt{N}}\int_{N^{-3/10}}^{\pi/2}
    \left|\frac{\frac{d}{d\theta}\tilde\phi(\theta)}{\sin(\theta)}\right|\,d\theta
    \leq
    C\frac{N^{3/10}}{\sqrt{N}}\int_{N^{-3/10}}^{\pi/2}
    \left|\frac{d}{d\theta}\tilde\phi(\theta)\right|\,d\theta\\
    &\leq
    \tilde CN^{3/10-1/2}\stackrel{N\to\infty}{\rightarrow} 0\,,
  \end{align*}
  where we have used  $\tilde\phi\in C^1([0,\pi])$.
		
  To treat $III_N$, we integrate by parts once more:
  \begin{align*}
    III_N
    &= -N^{-3/2}\int_{N^{-3/10}}^{\pi/2}iNe^{i (1 - \cos(\theta))\, N}\sin(\theta)
    \frac{\tilde\phi(\theta)\cos(\theta)}{\sin^3(\theta)}\, d\theta\displaybreak[2]\\
    &=-N^{-3/2}\left[e^{i (1 - \cos(\theta))\, N}
      \frac{\tilde\phi(\theta)\cos(\theta)}{\sin^3(\theta)}\right]_{\theta=N^{-3/10}}^{\pi/2}\\
    &\quad +N^{-3/2}\int_{N^{-3/10}}^{\pi/2}e^{i (1 - \cos(\theta))\, N}
    \frac{\frac{d}{d\theta}\tilde\phi(\theta)\cos(\theta)
      -\tilde\phi(\theta)\sin(\theta)}{\sin^3(\theta)}\,d\theta\\
    &\quad -N^{-3/2}\int_{N^{-3/10}}^{\pi/2}e^{i (1 - \cos(\theta))\, N}
    \frac{3\cos^2(\theta)\tilde\phi(\theta)}{\sin^4(\theta)}\,d\theta\,.
  \end{align*}
  Since $1/\sin^3(N^{-3/10}) \leq CN^{9/10}$, the term in square
  brackets is of order $N^{-3/2}N^{9/10}=N^{-3/5}\to 0$ as
  $N\to\infty$. For the integral expressions we note that
  $1/\sin^4(N^{-3/10}) \leq CN^{6/5}$ and by assumption $\tilde\phi\in
  C^1([0,\pi])$.  We conclude that the last integral scales as
  $N^{-3/2}N^{6/5}=N^{-3/10}\to 0$ as $N\to\infty$. The second
  integral is of lower order.  This proves that $III_N\to 0$.
 
  We now treat the other part of $A^N_\phi$, the integral over the
  interval $(0,N^{-3/10})$. We first note that, since $\tilde\phi$ is
  Lipschitz-continuous, for $\theta\in (0,N^{-3/10})$ one has
  $|\tilde\phi(\theta)-\tilde\phi(0)|\leq C\theta\leq CN^{-3/10}$ and
  thus
  \begin{equation*}
    \left|\int_0^{N^{-3/10}} N^{1/2}\, e^{i (1 - \cos(\theta))\, N}\,
      \tilde\phi(\theta)\, d\theta-\int_0^{N^{-3/10}} N^{1/2}\, e^{i (1 - \cos(\theta))\, N}\,
      \tilde\phi(0)\, d\theta\right|\leq CN^{1/2-3/5}\,,
  \end{equation*}
  which vanishes in the limit as $N\to \infty$.  In view of this
  smallness, it remains to investigate the integral
  \begin{equation}
	\label{eq:toconsiderinappendix}
    \tilde\phi(0)\,(2\pi i)^{-1/2}\int_0^{N^{-3/10}} N^{1/2}\, e^{i (1 - \cos(\theta))\, N}\, d\theta\,.
  \end{equation}
  As a result we find, using Lemma \ref{lem:OscillatoryAppendix} in
  the appendix and recalling that $\tilde\phi(0)=2\phi(e_1)$, as $N\to
  \infty$,
  \begin{align*}
    &\tilde\phi(0) (2\pi i)^{-1/2} \int_0^{N^{-3/10}} N^{1/2}\, e^{i (1 - \cos(\theta))\, N}\, d\theta\\
    &\qquad \to \tilde\phi(0)(2\pi i)^{-1/2} \frac12 (2\pi i)^{1/2} = \phi(e_1)\,.
  \end{align*}
  This shows the claim \eqref {eq:Dirac-aim1} in dimension $d=2$. 	
\end{proof}

\begin{appendix}
\section{An oscillatory integral}
 
We want to evaluate the limit of the integral
\eqref{eq:toconsiderinappendix}.

\begin{lemma}
  \label{lem:OscillatoryAppendix}
  Let $\beta\in (1/6,1/2)$. Then, as $N\to \infty$,
  \begin{equation}
    \label{eq:oscillatoryintegral1}
    I_N:=
    \int_0^{N^{-\beta}} N^{1/2}\, e^{i (1 - \cos(\theta))\, N}\, d\theta
    \to\frac12 \sqrt{\pi}(1+i) = \frac12 (2\pi i)^{1/2}\,.
  \end{equation}
\end{lemma}

\begin{proof}
  The integral in \eqref{eq:oscillatoryintegral1} can be written with
  the substitution $z = (1-\cos(\theta))N$, leading to $d\theta =
  dz/(N\sin(\theta))$. We find
  \begin{equation}
    I_N = \int_0^{(1-\cos(N^{-\beta}))N} e^{iz}
    \frac{1}{N^{1/2}\sin(\theta)}\, dz\,.
  \end{equation}
  Next we use the approximation $\frac{1}{N^{1/2}\sin(\theta)}\approx
  \frac{1}{\sqrt{2z}}$ and
  $(1-\cos(N^{-\beta}))N\approx\infty$. Indeed,
  \begin{equation*}
    CN^{1-2\beta}\leq (1-\cos(N^{-\beta}))N\leq \tilde CN^{1-2\beta}\,.
  \end{equation*}
  Since $\beta<1/2$ one finds $(1-\cos(N^{-\beta}))N\to\infty$ for $N\to \infty$. 
  Regarding the approximation of $\frac{1}{N^{1/2}\sin(\theta)}$ we obtain
  \begin{align*}
    \sin(\theta) &= \sqrt{1-\cos^2(\theta)}=\sqrt{1-(1-z/N)^2}\\
    &= \sqrt{\frac{2z}{N}}\sqrt{1-\frac{z}{2N}}
    = \sqrt{\frac{2z}{N}} + O\left((z/N)^{3/2}\right) 
  \end{align*}
  and thus, expanding the fraction,
  \begin{align*}
    \frac{1}{N^{1/2}\sin(\theta)}=\frac{1}{\sqrt{2z} +
      O\left(\sqrt{N}(z/N)^{3/2}\right)}=\frac{1}{\sqrt{2z}} +
    O\left(z^{1/2}N^{-1}\right)\,.
  \end{align*}
  Since in the domain of integration $z\leq \tilde CN^{1-2\beta}$, we finally find 
  \begin{align*}
     &\left|\int_0^{(1-\cos(N^{-\beta}))N} e^{iz}
      \left(\frac{1}{N^{1/2}\sin(\theta)}- \frac{1}{\sqrt{2z}}\right)\, dz\right|\\
    &\qquad \leq C\left(N^{1-2\beta}\right)^{1/2}N^{-1}\, N^{1-2\beta} = C
    N^{1/2-3\beta}\to 0
  \end{align*}
  as $N\to \infty$, since $\beta>1/6$. To sum up, we obtain
  \begin{align*}
    I_N &= \int_0^{(1-\cos(N^{-\beta}))N} e^{iz}
    \frac{1}{N^{1/2}\sin(\theta)}\, dz\\
    &\to \int_0^\infty
    e^{iz}\frac{1}{\sqrt{2z}}\, dz\, = \sqrt{2} \int_0^\infty e^{i
      p^2}\, dp = \frac12 \sqrt{\pi} (1+i) =\frac12 (2\pi i)^{1/2}\,.
  \end{align*}
  In the last line we used the substitution $z = p^2$ and Fresnel
  integrals: For real and imaginary part there holds $\int_0^\infty
  \sin(x^2)\, dx = \int_0^\infty \cos(x^2)\, dx =
  \sqrt{\pi}/(2\sqrt{2})$. This provides the claim of
  \eqref{eq:oscillatoryintegral1}.
\end{proof}

\end{appendix}

\bibliographystyle{abbrv}
\bibliography{lit_wave}

\begin{thebibliography}{10}

\bibitem{AbdulleP16}
A.~Abdulle and T.~Pouchon.
\newblock Effective models for the multidimensional wave equation in
  heterogeneous media over long time and numerical homogenization.
\newblock {\em Math. Models Methods Appl. Sci.}, 26(14):2651--2684, 2016.

\bibitem{MR2060593}
G.~Allaire.
\newblock Dispersive limits in the homogenization of the wave equation.
\newblock {\em Ann. Fac. Sci. Toulouse Math. (6)}, 12(4):415--431, 2003.

\bibitem{benoit:hal-01449353}
A.~Benoit and A.~Gloria.
\newblock {Long-time homogenization and asymptotic ballistic transport of
  classical waves}.
\newblock arXiv:1701.08600 and hal-01449353, Nov. 2016.

\bibitem{DLS14}
T.~Dohnal, A.~Lamacz, and B.~Schweizer.
\newblock Bloch-wave homogenization on large time scales and dispersive
  effective wave equations.
\newblock {\em Multiscale Model. Simul.}, 12(2):488--513, 2014.

\bibitem{DLS15}
T.~Dohnal, A.~Lamacz, and B.~Schweizer.
\newblock Dispersive homogenized models and coefficient formulas for waves in
  general periodic media.
\newblock {\em Asymptot. Anal.}, 93(1-2):21--49, 2015.

\bibitem{MR3733381}
J.~L. L\'{o}pez and J.~Soler.
\newblock A space-time {W}igner function approach to long time
  {S}chr\"{o}dinger-{P}oisson dynamics.
\newblock {\em SIAM J. Math. Anal.}, 49(6):4915--4941, 2017.

\bibitem{Nedelec-book}
J.-C. N\'{e}d\'{e}lec.
\newblock {\em Acoustic and electromagnetic equations}, volume 144 of {\em
  Applied Mathematical Sciences}.
\newblock Springer-Verlag, New York, 2001.
\newblock Integral representations for harmonic problems.

\bibitem{SantosaSymes91}
F.~Santosa and W.~W. Symes.
\newblock A dispersive effective medium for wave propagation in periodic
  composites.
\newblock {\em SIAM J. Appl. Math.}, 51(4):984--1005, 1991.

\bibitem{Schweizer-Theil-2017}
B.~Schweizer and F.~Theil.
\newblock Lattice dynamics on large time scales and dispersive effective
  equations.
\newblock {\em SIAM J. Appl. Math.}, 78(6):3060--3086, 2018.

\bibitem{Wong2001}
R.~Wong.
\newblock {\em Asymptotic approximations of integrals}, volume~34 of {\em
  Classics in Applied Mathematics}.
\newblock Society for Industrial and Applied Mathematics (SIAM), Philadelphia,
  PA, 2001.
\newblock Corrected reprint of the 1989 original.

\end{thebibliography}

\end{document}